\documentclass[11pt]{amsart}

\usepackage{amsmath}              
\usepackage{amsfonts}              
\usepackage{amsthm} 
\usepackage[pdftex]{graphicx}
\usepackage{mdwlist}
\usepackage{amssymb}
\usepackage{multicol}
\usepackage{pifont}
\usepackage{wrapfig}

\usepackage{setspace}

\usepackage{chngpage}
\usepackage{arydshln}

\usepackage{pdflscape}

\newtheorem{thm}{Theorem}[section]
\newtheorem{lem}[thm]{Lemma}
\newtheorem{prop}[thm]{Proposition}
\newtheorem{cor}[thm]{Corollary}

\newtheorem{procedure}[thm]{Procedure}%[section]

\theoremstyle{definition}

\theoremstyle{definition}
\newtheorem{defn}[thm]{Definition}
\newtheorem{question}[thm]{Question}
\newtheorem{observation}[thm]{Observation}

\theoremstyle{plain}

\allowdisplaybreaks[3]  %allow page breaks within large equations.

\begin{document}

\title[Symbolic Generic Initial Systems of Six Points]{The Asymptotics of Symbolic Generic Initial Systems of Six Points in $\mathbb{P}^2$}
\author{Sarah Mayes}
\address{Department of Mathematics, University of Michigan, 530 Church Street, Ann Arbor MI 48109}
\email{mayess@umich.edu}

\maketitle
\vspace*{-2em} %%reduce space between title and abstract

\normalsize
\begin{abstract}
Consider the ideal $I \subseteq K[x,y,z]$ corresponding to six points of $\mathbb{P}^2$.  We study the limiting behaviour of the \textit{symbolic generic initial system}, $\{ \text{gin}(I^{(m)}\}_m$ of $I$ obtained by taking the reverse lexicographic generic initial ideals of the uniform fat point ideals $I^{(m)}$.  The main result of this paper is a theorem describing the \textit{limiting shape} of $\{ \text{gin}(I^{(m)}\}_m$ for each of the eleven possible configuration types of six points.
\end{abstract}
\section{Introduction}

Given a set of six points  of $\mathbb{P}^{n-1}$ with ideal $I \subseteq k[\mathbb{P}^{n-1}]$, we may consider the ideal $I^{(m)}$ generated by the polynomials  that vanish to at least order $m$ at each of the points.  Such ideals are called \textit{uniform fat point ideals} and, although they are easy to describe, they have proven difficult to understand.  There are still many open problems and unresolved conjectures related to finding the Hilbert function of $I^{(m)}$ and even the degree $\alpha(I^{(m)})$ of the smallest degree element of $I^{(m)}$ (for example, see \cite{CHT11}, \cite{GH07}, \cite{GV04}, \cite{GHM09} , and \cite{Harbourne02}). 

In this paper we will study a \textit{limiting shape} that describes the behaviour of the Hilbert functions of the set of fat point ideals $\{ I^{(m)}\}_m$ as $m$ approaches infinity.  Studying asymptotic behaviour has been an important research trend of the past twenty years; while individual algebraic objects may be complicated, the limit of a collection of such objects is often quite nice (see, for example, \cite{Huneke92}, \cite{Siu01},\cite{ELS01}, and \cite{ES09}).  Research on fat point ideals has shown that certain challenges in understanding these ideals can be overcome by studying the entire collection $\{ I^{(m)}\}_m$.  For instance, more can be said about the limit
$\epsilon(I) = \lim_{m \rightarrow \infty} \frac{\alpha(I^{(m)})}{rm}$
than the invariants $\alpha(I^{(m)})$ of each ideal (see \cite{BC10} and \cite{Harbourne02}).  

To describe the limiting behaviour of the Hilbert functions of fat point ideals, we will study the \textit{symbolic generic initial system}, $\{ \text{gin}(I^{(m)}) \}_m$, obtained by taking the reverse lexicographic generic initial ideals of fat point ideals.  When $I \subseteq K[x,y,z]$ is an ideal of points of $\mathbb{P}^2$, knowing the Hilbert function of $I^{(m)}$ is equivalent to knowing the generators of $\text{gin}(I^{(m)})$; thus, describing the limiting behaviour of the symbolic generic initial system of $I$ is equivalent to describing that of the Hilbert functions of the fat point ideals $I^{(m)}$ as $m$ gets large.

We define the \textit{limiting shape} $P$ of the symbolic generic initial system $\{ \text{gin}(I^{(m)}\}_m$ of the ideal $I$ to be the limit $\lim_{m \rightarrow \infty} \frac{1}{m} P_{\text{gin}(I^{(m)})}$, where $P_{\text{gin}(I^{(m)})}$ denotes the Newton polytope of $\text{gin}(I^{(m)})$.  When $I \subseteq K[x,y,z]$ corresponds to an arrangement of points in $\mathbb{P}^2$, each of the ideals $\text{gin}(I^{(m)})$ is generated in the variables $x$ and $y$, so $P_{\text{gin}(I^{(m)})}$, and thus $P$, can be thought of as a subset of $\mathbb{R}^2$.

The main result of this paper is the following theorem describing the limiting shape of the symbolic generic initial system of an ideal corresponding to any collection of 6 points in $\mathbb{P}^2$.  The concept of \textit{configuration type} mentioned is intuitive; for example, $\{ p_1, \dots, p_6 \}$ are of configuration type B pictured in Figure \ref{fig:AthroughF} when there is one line through three of the points but no lines through any other three points and no conics through all six points (see Definition \ref{defn:configtype}).

\begin{thm}
\label{thm:mainthm}
Let $I \subseteq K[x,y,z]$ be the ideal corresponding to a set of six points in $\mathbb{P}^2$.  Then the limiting polytope $P$ of the reverse lexicographic symbolic generic initial system $\{ \textnormal{gin}(I^{(m)} )\}_m$ is equal to the limiting shape $P$ shown in Figures \ref{fig:AthroughF} and \ref{fig:GthroughK} corresponding to the configuration type of the six points.
\end{thm}

This theorem will be proved in Section \ref{sec:proof}; Sections \ref{sec:background} and \ref{sec:HilbFn} contain background information necessary for the proof.  In Section \ref{sec:questions} we discuss how characteristics of the arrangement of an arbitrary set of points in $\mathbb{P}^2$ are, or may be, reflected in the limiting shape of the corresponding symbolic generic initial system.

	\begin{landscape}

\begin{figure}
\centering
\includegraphics[height=13cm]{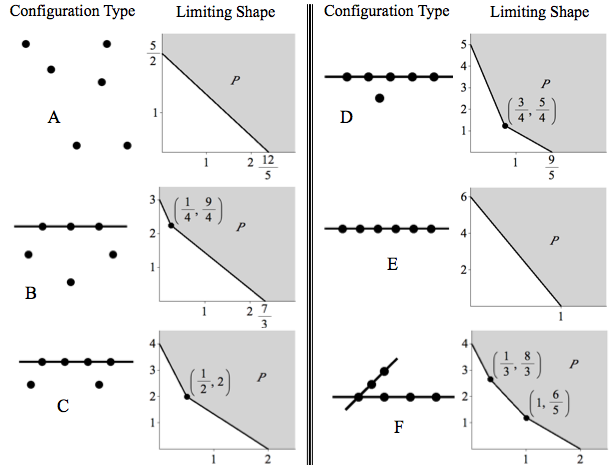}
\caption{The limiting shape $P$ of the generic initial systems $\{ \text{gin}(I^{(m)}) \}_m$ when $I$ is the ideal corresponding to points $\{p_1, \dots, p_r \}$ in configuration types A through F pictured.}
\label{fig:AthroughF}
\end{figure}

\begin{figure}
\centering
\includegraphics[height=13cm]{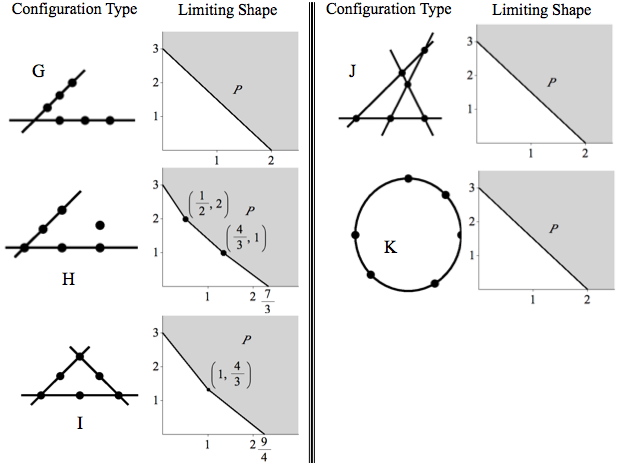}
\caption{The limiting shape $P$ of the generic initial systems $\{ \text{gin}(I^{(m)}) \}_m$ when $I$ is the ideal corresponding to points $\{p_1, \dots, p_r \}$ in configuration types G through K pictured.}
\label{fig:GthroughK}
\end{figure}

\end{landscape}
\section{Background}
\label{sec:background}

In this section we will introduce notation, definitions, and results related to fat points in $\mathbb{P}^2$, generic initial ideals, and systems of ideals. Unless stated otherwise, $R=K[x,y,z]$ is the polynomial ring in three variables over a field $K$ of characteristic 0 with the standard grading and the reverse lexicographic order $>$ with $x>y>z$.

%%%%%Fat Points%%%%

\subsection{Fat Points in $\mathbb{P}^2$}
\label{sec:fatpoints}

\begin{defn}
Let $p_1, \dots, p_r$ be distinct points of $\mathbb{P}^2$, $I_j$ be the ideal of $K[\mathbb{P}^2] = R$ consisting of all forms vanishing at the point $p_j$, and $I = I_1 \cap \cdots \cap I_r$ be the ideal of the points $p_1, \dots, p_r$.  A \textbf{fat point subscheme} $Z = m_1p_1 + \cdots + m_rp_r$, where the $m_i$ are nonnegative integers, is the subscheme of $\mathbb{P}^2$ defined by the ideal $I_Z = I_1^{m_1} \cap \cdots \cap I_r^{m_r}$ consisting of forms that vanish at the points $p_i$ with multiplicity at least $m_i$.  When $m_i = m$ for all $i$, we say that $Z$ is \textbf{uniform}; in this case, $I_Z$ is equal to the $m^{\text{th}}$ symbolic power of $I$, $I^{(m)}$.
\end{defn}

The following lemma relates the symbolic and ordinary powers of $I$ in the case we are interested in (see, for example, Lemma 1.3 of \cite{AV03}).

\begin{lem} %cite a source for this.
\label{lem:symbpower}
If $I$ is the ideal of distinct points in $\mathbb{P}^2$,
$$(I^m)^{\textnormal{sat}} = I^{(m)},$$
where $J^{\textnormal{sat}} = \bigcup_{k \geq 0}(J:\mathfrak{m}^k)$ denotes the saturation of $J$.
\end{lem}

The precise definition of a \textit{configuration type} mentioned in the statement of Theorem \ref{thm:mainthm} is as follows.

\begin{defn}[{\cite{GH07}}]
\label{defn:configtype}
Two sets of points $\{ p_1, \dots, p_r \}$ and $\{ p'_1, \dots, p'_r \}$ of $\mathbb{P}^2$ have the same \textbf{configuration type} if for all sequences of positive integers $m_1, \dots, m_r$ the ideals of the fat point subschemes $Z = m_1p_1+\cdots + m_rp_r$ and $Z' = m_1p'_1+\cdots+m_rp'_r$ have the same Hilbert function, possibly after reordering.
\end{defn}

\begin{prop}[{\cite{GH07}}]
\label{prop:AllConfigsCovered}
The configuration types for six distinct points in $\mathbb{P}^2$ are exactly the configurations A through K shown in Figures \ref{fig:AthroughF} and \ref{fig:GthroughK}.
\end{prop}

%%Generic Initial Ideals 
\subsection{Generic Initial Ideals}

An element $g = (g_{ij}) \in \text{GL}_n(K)$ acts on $R = K[x_1, \dots, x_n]$ and sends any homogeneous element $f(x_1, \dots, x_n)$ to the homogeneous element 
$$f(g(x_1), \dots, g(x_n))$$ 
where $g(x_i) = \sum_{j=1}^n g_{ij}x_j$.  If $g(I)=I$ for every upper triangular matrix $g$ then we say that $I$ is \textit{Borel-fixed}.  Borel-fixed ideals are \textit{strongly stable} when $K$ is of characteristic 0; that is, for every monomial $m$ in the ideal such that $x_i$ divides $m$, the monomials $\frac{x_jm}{x_i}$ are also in the ideal for all $j<i$.  This property makes such ideals particularly nice to work with.

To any homogeneous ideal $I$ of $R$ we can associate a Borel-fixed monomial ideal $\text{gin}_{>}(I)$ which can be thought of as a coordinate-independent version of the initial ideal.  Its existence is guaranteed by Galligo's theorem (also see \cite[Theorem 1.27]{Green98}).

\begin{thm}[{\cite{Galligo74} and \cite{BS87b}}]
\label{thm:galligo}
For any multiplicative monomial order $>$ on $R$ and any homogeneous ideal $I\subset R$, there exists a Zariski open subset $U \subset \textnormal{GL}_n(K)$ such that $\textnormal{In}_{>}(g(I))$ is constant and Borel-fixed for all $g \in U$.  
\end{thm}

\begin{defn}
The \textbf{generic initial ideal of $I$}, denoted $\text{gin}_{>}(I)$, is defined to be $\text{In}_{>}(g(I))$ where $g \in U$ is as in Galligo's theorem.
\end{defn}

The \textit{reverse lexicographic order} $>$ is a total ordering on the monomials of $R$ defined by: 
\begin{enumerate}
\item if $|I| =|J|$ then $x^I > x^J$ if there is a $k$ such that $i_m = j_m$ for all $m>k$ and $i_k < j_k$; and
\item if $|I| > |J|$ then $x^I >x^J$.
\end{enumerate}
For example, $x_1^2 >x_1x_2 > x_2^2>x_1x_3>x_2x_3>x_3^2$.   From this point on, $\text{gin}(I) = \text{gin}_{>}(I)$ will denote the generic initial ideal with respect to the reverse lexicographic order.

Recall that the Hilbert function $H_I(t)$ of $I$ is defined by $H_I(t) = \text{dim}(I_t)$.  The following result is a consequence of the fact that Hilbert functions are invariant under making changes of coordinates and taking initial ideals (\cite{Green98}).

\begin{prop}
\label{prop:commonproperties}
For any homogeneous ideal $I$ in $R$, the Hilbert functions of $I$ and $\textnormal{gin}(I)$ are equal.
\end{prop}

We now describe the structure of the ideals $\text{gin}(I^{(m)})$ where $I$ is an ideal corresponding to points in $\mathbb{P}^2$.  The proof of this result is contained in \cite{Mayes12c} and follows from results of Bayer and Stillman (\cite{BS87}) and of Herzog and Srinivasan (\cite{HS98})

\begin{prop}[{Corollary 12.9 of \cite{Mayes12c}}]
\label{prop:formofgin}
Suppose $I \subseteq K[x,y,z]$ is the ideal of distinct points in $\mathbb{P}^2$.  Then the minimal generators of $\textnormal{gin}(I^{(m)})$ are 
$$\{ x^{\alpha(m)}, x^{\alpha(m)-1}y^{\lambda_{\alpha(m)-1}}, \dots, xy^{\lambda_1(m)}, y^{\lambda_0(m)} \}$$
for $\lambda_0(m), \dots, \lambda_{\alpha(m)-1}$ such that $\lambda_0(m) > \lambda_1(m) > \cdots > \lambda_{\alpha(m)-1}(m) \geq 1.$
\end{prop}

Since Borel-fixed ideals generated in two variables are determined by their Hilbert functions (see, for example,  Lemma 3.7 of \cite{Mayes12b}), we have the following corollary of Propositions
\ref{prop:commonproperties} and \ref{prop:formofgin}.

\begin{cor}
\label{cor:GISConfigType}
If $I$ and $I'$ are ideals corresponding to two point arrangements of the same configuration type, $\textnormal{gin}(I^{(m)}) = \textnormal{gin}(I'^{(m)})$ for all $m$.
\end{cor}

Actually finding the Hilbert functions of fat point ideals is not easy and is a significant area of research. (for example, see \cite{CHT11}, \cite{GH07}, \cite{GV04}, \cite{GHM09} , and \cite{Harbourne02})  When $I$ is the ideal of less than 9 points, however, techniques exist for computing these Hilbert functions.  In Section \ref{sec:HilbFn} we will outline the method used in this paper, following \cite{GH07}.  Other techniques, such as those in \cite{CHT11}, can also be used for some of the point arrangements A through K.

%%%Generic Initial Systems%%%

\subsection{Graded Systems}

In this subsection we introduce the limiting shape of a graded system of monomial ideals.

\begin{defn}[\cite{ELS01}]
A \textbf{graded system of ideals} is a collection of ideals $J_{\bullet} =\{J_i\}_{i=1}^{\infty}$ such that 
$$J_i \cdot J_j \subseteq J_{i+j} \hspace{0.3in} \text{ for all } i, j \geq 1.$$
\end{defn}

\begin{defn}
The \textbf{generic initial system} of a homogeneous ideal $I$ is the collection of ideals $J_{\bullet}$ such that $J_i = \text{gin}(I^i)$.  The \textbf{symbolic generic initial system} of a homogeneous ideal $I$ is the collection $J_{\bullet}$ such that $J_i = \text{gin}(I^{(i)})$.
\end{defn}

The following lemma justifies calling these collections `systems'; see Lemma 2.5 of \cite{Mayes12a} and Lemma 2.2 of \cite{Mayes12c} for proofs.

\begin{lem}
Generic initial systems and symbolic generic initial systems are graded system of ideals.
\end{lem}

Let $J$ be a monomial ideal of $R=K[x_1,\dots,x_n]$.  We may associate to $J$ a subset $\Lambda$ of $\mathbb{N}^n$ consisting of the points $\lambda$ such that $x^{\lambda} \in J$.  The \textit{Newton polytope} $P_J$ of $J$ is the convex hull of $\Lambda$ regarded as a subset of $\mathbb{R}^n$.  Scaling the polytope $P_J$ by a factor of $r$ gives another polytope that we will denote $rP_J$.

If $\mathrm{a}_{\bullet}$ is a graded system of monomial ideals in $R$, the polytopes of $\{ \frac{1}{q} P_{\mathrm{a}_q} \}_q$ are nested: $\frac{1}{c}P_{\mathrm{a}_c} \subset \frac{1}{c+1}P_{\mathrm{a}_{c+1}}$ for all $c \geq 1$.  The \textit{limiting shape $P$ of $\mathrm{a}_{\bullet}$} is the limit of the polytopes in this set:
$$P = \bigcup_{q \in \mathbb{N}^*} \frac{1}{q} P_{\mathrm{a}_q}.$$

When $I$ is the ideal of points in $\mathbb{P}^2$ $\text{gin}(I^{(m)})$ is generated in the variables $x$ and $y$ by Proposition \ref{prop:formofgin}, so we can think of each $P_{\text{gin}(I^{(m)})}$, and thus $P$, as a subset of $\mathbb{R}^2$.

\section{Technique for computing the Hilbert function}
\label{sec:HilbFn}

Here we summarize the method that is used to compute $H_{I^{(m)}}(t)$ in this paper.  It follows the work of Guardo and Harbourne in \cite{GH07}; details can be found there.

Suppose that $\pi: X \rightarrow \mathbb{P}^2$ is the blow-up of distinct points $p_1, \dots, p_r$ of $\mathbb{P}^2$. Let $E_i = \pi^{-1}(p_i)$ for $i = 1, \dots, r$ and $L$ be the total transform in $X$ of a line  not passing through any of the points $p_1, \dots, p_r$.  The classes of these divisors form a basis of $\text{Cl}(X)$; for convenience, we will write $e_i$ in place of $[E_i]$ and $e_0$ in place of $[L]$.  Further, the intersection product in $\text{Cl}(X)$ is defined by $e_i^2 = -1$ for $i=1, \dots, r$; $e_0^2 = 1$; and $e_i \cdot e_j = 0$ for all $i\neq j$.

Let $Z = m(p_1+\cdots +p_r)$ be a uniform fat point subscheme with sheaf of ideals $\mathcal{I}_Z$; set 
$${F}_d = dE_0 - m(E_1 + E_2 + \cdots +E_r)$$
and $\mathcal{F}_d = \mathcal{O}_X(F_d)$.  

The following lemma relates divisors on $X$ to the Hilbert function of $I^{(m)}$.

\begin{lem}
\label{lem:relationWithDivisors}
If $F_d = dE_0-m(E_1+\cdots +E_r)$ then $h^0(X, \mathcal{F}_d) = H_{I^{(m)}}(d)$.
\end{lem}

\begin{proof}
Since $\pi_{*}(\mathcal{F}_d) = \mathcal{I}_Z \otimes \mathcal{O}_{\mathbb{P}^2}(d)$, 
$$H_{I^{(m)}}(d) = \text{dim}((I_Z)_d) = h^0(\mathbb{P}^2, \mathcal{I}_Z \otimes \mathcal{O}_{\mathbb{P}^2}(d)) = h^0(X, \mathcal{F}_d)$$
for all $d$. 
\end{proof}

For convenience, we will sometimes write $h^0(X, F) = h^0(X, \mathcal{O}_X(F))$.  Recall that if  $[{F}]$ not the class of an effective divisor then $h^0(X, {F}) = 0$.  On the other hand, if $F$ is effective, then we will see that we can compute $h^0(X,{F})$ by finding $h^0(X,{H})$ for some \textit{numerically effective} divisor $H$.  %A reference to this is p. 3 of "Resolutions of Fat Point Ideals in P^2$ paper.

\begin{defn}
A divisor $H$ is \textbf{numerically effective} if $[F] \cdot [H] \geq 0$ for every effective divisor $F$, where $[F] \cdot [H]$ denotes the intersection multiplicity.  The cone of classes of numerically effective divisors in $\text{Cl}(X)$ is denoted by NEF($X$).
\end{defn}

\begin{lem}
\label{lem:h0ofNEFF}
Suppose that $X$ is the blow-up of $\mathbb{P}^2$ at $r \leq 8$ points in general position and that $F \in \text{NEF}(X)$.  Then $F$ is effective and 
$$h^0(X, F)  = ([F]^2-[F]\cdot [K_X])/2+1$$
where $K_X = -3E_0 + E_1 + \cdots + E_r$.
\end{lem}

\begin{proof}
This is a consequence of Riemann-Roch and the fact that $h^1(X, F) = 0$ for any numerically effective divisor $F$. See Lemma 2.1b of \cite{GH07} for a discussion.
\end{proof}

The set of classes of effective, reduced, and irreducible curves of negative intersection is
$$\text{NEG}(X) := \{ [C] \in \text{Cl}(X) : [C]^2 <0, C \text{ is effective, reduced, and irreducible}\}.$$
The set of classes in $\text{NEG}(X)$ with self intersection less than $-1$ is
$$\text{neg}(X) := \{[C] \in \text{NEG}(X) : [C]^2<-1 \}.$$

The following result of Guardo and Harbourne allows us to easily identify divisor classes belonging to $\text{NEG}(X)$.  In the lemma, the \textit{curves defining the configuration type} are lines that pass through any three points or conics that pass through any six points.  For example, the divisors defining the configuration type shown in Figure \ref{fig:6HArrangement} are $E_0-E_1-E_2-E_3$ and $E_0-E_1-E_4-E_5$.  

\begin{lem}[{Lemma 2.1d of \cite{GH07}}]
\label{lem:NegElements}
The elements of $\text{neg}(X)$ are the classes of divisors that correspond to the curves defining the configuration types.  Further, 
$$\text{NEG}(X) = \text{neg}(X) \cup \{ [C] \in \mathcal{B} \cup \mathcal{L} \cup \mathcal{Q} : [C]^2 = -1, [C] \cdot [D] \geq 0 \text{ for all } D \in \text{neg}(X) \}$$
where $\mathcal{B} = \{e_i : i >0 \}$, $\mathcal{L} = \{ e_0 - e_{i_1}- \cdots -e_{i_r} : r \geq 2, 0 <i_1< \cdots < i_r \leq 6 \}$, and $\mathcal{Q} = \{ 2e_0 - e_{i_1} - \cdots - e_{i_r} : r \geq 5, 0 < i_1 < \cdots < i_r \leq 6 \}$.
\end{lem}

\begin{figure}
\centering
\includegraphics[height=2cm]{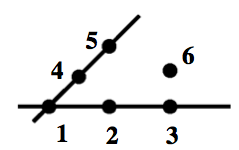}
\caption{Points $p_1, \dots, p_6$ of configuration type H.}
\label{fig:6HArrangement}
\end{figure}

The following result will be used in Procedure \ref{proc:findH}; see Section 2 of \cite{GH07}.

\begin{lem}
\label{lem:Equalh0}
Suppose that $[C]  \in \text{NEG}(X)$ is such that $[F] \cdot [C] <0$.  Then $h^0(X, F) = h^0(X, F-C)$.
\end{lem}

Knowing how to compute $h^0(X, H)$ for a numerically effective divisor $H$ will allow us to compute $h^0(X, F)$ for \textit{any} divisor $F$.  In particular, given a divisor $F$, there exists a divisor $H$ such that $h^0(X, F) = h^0(X, H)$ and either: 
\begin{enumerate}
\item[(a)]  $H$ is numerically effective so $$h^0(X, F) = h^0(X, H) = (H^2-H\cdot K_X)/2+1$$ by Lemma \ref{lem:h0ofNEFF}; or
\item[(b)]  there is a numerically effective divisor $G$ such that $[G]\cdot [H] <0$ so $[H]$ is not the class of an effective divisor and $h^0(X, F) = h^0(X, H) = 0$.
\end{enumerate}

The method for finding such an $H$ is as follows. 

\begin{procedure}[{Remark 2.4 of \cite{GH07}}]
\label{proc:findH}
Given a divisor $F$ we can find a divisor $H$ with $h^0(X, F) = h^0(X, H)$ satisfying either condition (a) or (b) above as follows.
\begin{enumerate}
\item Reduce to the case where $[F] \cdot e_i \geq 0$ for all $i=1, \dots, n$:  if $[F]\cdot e_i <0$ for some $i$, $h^0(X, F) = h^0(X, F-([F]\cdot e_i)E_i)$, so we can replace $F$ with $F - ([F]\cdot e_i) E_i$.
\item Since $L$ is numerically effective, if $[F]\cdot e_0<0$ then $[F]$ is not the class of an effective divisor and we can take $H=F$ (case (b)).
\item If $[F] \cdot [C] \geq 0$ for every $[C] \in \text{NEG}(X)$ then, by Lemma \ref{lem:NegElements}, $F$ is numerically effective, so we can take $H = F$ (case (a)).
\item If $[F] \cdot [C] <0$ for some $[C]\in \text{NEG}(X)$ then $h^0(X, F) = h^0(X, F-C)$ by Lemma \ref{lem:Equalh0}. Then replace $F$ with $F-C$ and repeat from Step 2.  
\end{enumerate}
\end{procedure}

There are only a finite number of elements in $\text{NEG}(X)$ to check by Lemma \ref{lem:NegElements} so it is possible to complete Step 3.  Further, $[F] \cdot e_0 > [F-C] \cdot e_0$ when $[C] \in \text{NEG}(X)$, so the condition in Step 2 will be satisfied after at most $[F]\cdot e_0 +1$ repetitions.  Thus, the process will terminate.

Taking these results together we can compute the Hilbert function of $I^{(m)}$ as follows.
\begin{enumerate}
\item Compute $\text{NEG}(X)$ from $\text{neg}(X)$ using Lemma \ref{lem:NegElements}.
\item Find $H_t$ corresponding to $F_t$ using Procedure \ref{proc:findH} for all $t$.
\item Compute $H_{I^{(m)}}(t) = h^0(X, F_t) = h^0(X, H_t)$ with Lemma \ref{lem:h0ofNEFF}.
\end{enumerate}

\section{Proof of the Main Theorem}
\label{sec:proof}

In this section, we will outline the proof of Theorem \ref{thm:mainthm}.
Recall that ideals of points with the same configuration type have the same symbolic generic initial system by Corollary \ref{cor:GISConfigType} so the statement of the theorem makes sense.  Further, Proposition \ref{prop:AllConfigsCovered} ensures that the theorem includes all possible sets of six points.

If $I$ is the ideal of a set of six points having configuration type $E$, $G$, or $K$, the theorem follows from the main result of \cite{Mayes12a}.  Likewise, if $I$ is the ideal of six points of configuration type $A$, the theorem follows from the main result of \cite{Mayes12c}.

For the remaining cases we can find the limiting polytope of $\{ \text{gin}(I^{(m)}) \}_m$ by following the five steps below.  First, we record a lemma that will be used in Step 2.

\begin{lem}
\label{lem:NumberInDegree}
Let $J$ be a monomial ideal of $K[x,y,z]$ generated in the variables $x$ and $y$.  Then the number of elements of $J$ of degree $t$ only involving the variables $x$ and $y$ is equal to $H_J(t) - H_J(t-1)$.  The number of minimal generators of $J$ in degree $t$ is equal to $H_J(t) - H_J(t-2)-1$.
\end{lem}

\begin{proof}
The first statement follows from the fact that there are exactly $H_J(t-1)$ monomials of $J$ of degree $t$ involving the variable $z$.  The number of generators in degree $t$ is equal to the number of monomials of $J$ in the variables $x$ and $y$ of degree $t$ minus the number of monomials of $J$ that arise from multiplying the elements of degree $t-1$ in $x$ and $y$ by the variables $x$ and $y$.  Using this, the last statement follows from the first.
\end{proof}

\begin{enumerate}

\item[\textbf{Step 1:}]  Find the Hilbert function of $I^{(m)}$ for infinitely many $m$ by using the method outlined in Section \ref{sec:HilbFn}.

\item[\textbf{Step 2:}]  Find the number of minimal generators of $\text{gin}(I^{(m)})$ of each degree $t$ for infinitely many $m$.  We can use Lemma \ref{lem:NumberInDegree} for this computation because $\text{gin}(I^{(m)})$ is an ideal generated in the variables $x$ and $y$ (Proposition \ref{prop:formofgin}) and we know the Hilbert function of $\text{gin}(I^{(m)})$ from Proposition \ref{prop:commonproperties} and Step 1.

\item[\textbf{Step 3:}]  Write down the generators of $\text{gin}(I^{(m)})$ for infinitely many $m$.  Note that this follows from Step 2 since 
$$\text{gin}(I^{(m)}) = (x^{\alpha(m)}, x^{\alpha(m)-1}y^{\lambda_{\alpha(m)-1}}, \dots, xy^{\lambda_1(m)}, y^{\lambda_0(m)})$$
where $\lambda_0(m) > \cdots > \lambda_{k-1}(m) \geq 1$ by Proposition \ref{prop:formofgin}.

\item[\textbf{Step 4:}]  Compute the Newton polytope $P_{\text{gin}(I^{(m)})}$ of each $\text{gin}(I^{(m)})$ for infinitely many $m$.  Recall that the boundary of these polytopes is determined by the convex hull of the points $(i, \lambda_i(m))$ and $(\alpha(m),0)$.

\item[\textbf{Step 5:}]  Find the limiting polytope of the symbolic generic initial system of $I$.  To do this it suffices to take the limit
$$P = \bigcup_{m \in \mathbb{N}^*} \frac{1}{m} P_{\mathrm{a}_m}$$
over an infinite subset of $\mathbb{N}^*$.

\end{enumerate}

All of the remaining calculations are similar but long so, for the sake of space, we will only record the proof here for configuration H.

	\subsection{Proof of main theorem for configuration H}

Let $I$ be the ideal of points $p_1, \dots, p_6$ of configuration type H, ordered as in Figure \ref{fig:6HArrangement}.

%%Figure used to be here

\noindent \textbf{Step 1.}

First we will follow the method outlined in Section 3 to find $H_{I^{(m)}}$ for infinitely many $m$.  We will use the notation from Section \ref{sec:HilbFn} and will often denote the divisor $a_0E_0-(a_1E_1+a_2E_2+a_3E_3+a_4E_4+a_5E_5+a_6E_6)$ by $(a_0; a_1, a_2, a_3, a_4, a_5, a_6)$.  Also, if $F_1$ and $F_2$ are divisors, $F_1 \cdot F_2$ denotes $[F_1]\cdot[F_2]$, the intersection multiplicity of their classes.

First we need to determine $\text{NEG}(X)$.  Note that the configuration type H is defined by a line through points 1, 2, and 3 and another line through points 1, 4, and 5.  Thus, $\text{neg}(X)$ consists of the classes of $ A_1:=E_0-E_1-E_2-E_3$ and $A_2:=E_0-E_1-E_4-E_5$.
The other elements of $\text{NEG}(X)$ are exactly those $[C] \in \mathcal{B} \cup \mathcal{L} \cup \mathcal{Q}$ such that $[C]^2 =-1$  and $[C] \cdot [D] \geq 0$ for all $[D] \in \text{neg}(X)$ by Lemma \ref{lem:NegElements}.  Using this, one can check that $\text{NEG}(X)$ consists of the classes of the divisors
\begin{eqnarray*}
 &&A_1:=E_0-E_1-E_2-E_3, A_2:=E_0-E_1-E_4-E_5, \\
 &&B := E_0 - E_1-E_6, \\
&&C_i := E_0 - E_i-E_6 \text{    for } i = 2, 3, 4, 5, \\
&&D_{ij} :=E_0-E_i-E_j \text{    for } i = 2, 3 \text{ and } j=4,5, \\
 &&Q:=2E_0-E_2-E_3-E_4-E_5-E_6.
\end{eqnarray*}

Next, we will follow Procedure \ref{proc:findH} for each $F_t$ once we fix $m$ divisible by 12.  The procedure produces a divisor $H_t$ that is either numerically effective or is in the class of an effective divisor such that
$$H_{I^{(m)}}(t) = h^0(X, F_t) = h^0(X, H_t).$$
First, we will make some observations about which elements of $\text{NEG}(X)$ may be subtracted during the procedure.

Suppose that $J$ is a divisor of the form $J:=(a; b, c, c, c, c, d)$.   We will show that if the procedure allows us to subtract one $A_i$ (respectively, one $C_i$ or one $D_{ij}$) from $J$, we can subtract them all consecutively.  This is equivalent to showing that if the intersection multiplicity of $J$ with $A_1$ is negative then the intersection multiplicity of $J-A_1$ with $A_2$ is also negative; parallel statements hold for the $C_i$ and $D_{ij}$.

$\mathbf{A_i}:$ 
\begin{eqnarray*}
J \cdot A_1 &=& a-b-2c \\
(J-A_1)\cdot A_2 &=& (a-1; b-1, c-1, c-1, c, c, d) \cdot A_2 \\
				&=& a-1-(b-1)-2c = a-b-2c\\
\end{eqnarray*}

$\mathbf{C_i}:$ 
\begin{eqnarray*}
J \cdot C_2 &=& a-c-d\\
(J-C_2) \cdot C_3 &=& (a-1; b, c-1, c, c, c, d-1) \cdot C_3\\
				& =& (a-1)-c-(d-1) = a-c-d\\
(J-C_2-C_3) \cdot C_4 &=& (a-2; b, c-1,c-1,c,c,d-2) \cdot C_4 \\
				&=& (a-2) - c- (d-2) = a-c-d\\
(J-C_2-C_3-C_4) \cdot C_5 &=& (a-3; b; c-1, c-1, c-1, c, d-3) \cdot C_5 \\
				&=& (a-3)-c-(d-3) = a-c-d
\end{eqnarray*}

$\mathbf{D_{ij}}:$ 
\begin{eqnarray*}
J \cdot D_{24} &=& a-2c\\
(J-D_{24}) \cdot D_{25} &=& (a-1; b, c-1, c, c-1, c, d) \cdot D_{25} \\
				&=& (a-1)-(c-1)-c = a-2c\\
(J-D_{24}-D_{25}) \cdot D_{34} &=& (a-2; b, c-2, c, c-1, c-1, d) \cdot D_{34}\\
				& =& (a-2)-c-(c-1) = a-2c-1\\
(J-D_{24}-D_{25}-D_{34}) \cdot D_{35} &=& (a-3; b, c-2, c-1, c-2, c-1, d) \cdot D_{35}\\
				& =& (a-3)-2(c-1) = a-2c-1
\end{eqnarray*}

Define 
$$A:=A_1+A_2,  \hspace{0.1in} C :=C_2+C_3+C_4+C_5,  \hspace{0.1in} D:=D_{24}+D_{25}+D_{34}+D_{35}.$$
The calculations above show that if $J\cdot A_1 <0$ (if $J \cdot  C_2 <0$, $J \cdot D_{24} <0$, respectively) then the procedure will allow us to subtract one entire copy of $A$ ($C$, $D$). 
If we begin with a divisor of the form $J = (a; b, c, c, c, c, d)$ then $J-A$, $J-B$, $J-C$, $J-D$, and $J-Q$ have the same form.  These facts taken together mean that  that $H_t$ is obtained from $F_t$ - a divisor with the same form as $J$ -  by subtracting off copies of  $A$, $B$, $C$, $D$, and $Q$.

In Procedure \ref{proc:findH}, the requirement for being able to subtract an element of $\text{NEG}(X)$ from $J$ is that the intersection of that element with $J$ is strictly negative.  Thus, it is of interest how the intersection multiplicities with elements of $\text{NEG}(X)$ \textit{change} as other elements of $\text{NEG}(X)$ are subtracted from a divisor of the form $(a; b, c, c, c, c, d)$.

If $J= (a; b, c, c, c, c, d)$ as above, we have the following.

\begin{center}

\begin{tabular}{|r| c c c c c|}
\hline
& \multicolumn{5}{|c|}{value of $G$}\\
 &$A_i$&$B$&$C_i$&$D_{ij}$&$Q$\\
 \hline
 $(J-A)\cdot G - J\cdot G$ &2&0&-1&0&0\\
 $(J-B)\cdot G - J\cdot G$ &0&1&0&-1&-1\\
  $(J-C)\cdot G - J\cdot G$ &-2&0&1&-2&0\\
 $(J-D)\cdot G - J\cdot G$ &0&-4&-2&0&0\\
 $(J-Q)\cdot G - J\cdot G$ &0&-1&0&0&1\\
 \hline
\end{tabular}
\end{center}

We now use this set-up to obtain $H_t$ from $F_t$ by successively subtracting elements of $\text{NEG}(X)$ that have negative intersection with the remaining divisor.  First note that 
$$F_t \cdot A_i = t-3m< 0 \iff t<3m,$$ 
$$F_t \cdot B = F_t \cdot C_i = F_t \cdot D_{ij} = t-2m< 0 \iff t<2m,$$
and 
$$F_t \cdot Q = 2t-5m <0 \iff t<\frac{5m}{2}.$$

Therefore, $[F_t] = [H_t]$ (that is, $F_t$ is numerically effective) if and only if $t \geq 3m$.  In this case, $h^0(X, F_t) = \frac{1}{2}t^2-3m^2+\frac{3}{2}t-3m+1$ by Lemma \ref{lem:h0ofNEFF}.

We will assume from this point on that $12 | m$. 

Now suppose that $3m>t\geq \frac{5m}{2}$.  In this case, $[A_i] \cdot [F_t] <0$, but $[C] \cdot [F_t] \geq 0$ for all other $[C] \in \text{NEG}(X)$; thus, Procedure \ref{proc:findH} allows us to subtract $A_i$ - and thus $A$ - but no other divisors initially.  How many copies can we subtract?  From the table, we see that the intersection multiplicity of the remaining divisor with $A_i$ increases by 2 each time we subtract a copy of $A_i$.  We can keep subtracting copies of $A$ as long as the intersection multiplicity with $A_i$ is strictly negative; thus, we can subtract exactly 
$$\Bigg \lceil -\frac{F_t \cdot A_i}{2} \Bigg \rceil = \Bigg \lceil \frac{3m-t}{2} \Bigg \rceil$$
copies of $A$.  The only other intersection multiplicity that changes through the process subtracting $A$s is with the $C_i$, which decreases by one for each copy of $A$ subtracted.  Thus, 
$$\Bigg(F_t- \Bigg \lceil \frac{3m-t}{2} \Bigg \rceil A \Bigg) \cdot C_i = t-2m- \Bigg \lceil \frac{3m-t}{2} \Bigg\rceil$$
and this is never negative when $t \geq \frac{5m}{2}$ ($t$ must be at most $\frac{7m}{3}$ for this expression to be negative).
Thus, the intersection multiplicity of $F_t - \big \lceil \frac{3m-t}{2} \big \rceil$with all $[C] \in \text{NEG}(X)$ is nonnegative, so 
$$H_t = \Bigg( t-2\Bigg \lceil \frac{3m-t}{2} \Bigg \rceil; m- 2 \Bigg \lceil \frac{3m-t}{2} \Bigg \rceil, m- \Bigg \lceil \frac{3m-t}{2} \Bigg \rceil, \dots, m \Bigg).$$
When $t$ is even,
$$H_t = \Big( 2t-3m; t-2m, \frac{t-m}{2}, \dots, m \Big)$$
and $h^0(X, F_t) = t^2-3tm-\frac{3}{2}m^2+\frac{3}{2}t-3m+1$ while when $t$ is odd 
$$H_t = \Big( 2t-3m-1; t-2m-1, \frac{t-m-1}{2}, \dots, m \Big)$$
and $h^0(X, F_t) = t^2-3tm-\frac{3}{2}m^2+\frac{3}{2}t-3m+\frac{1}{2}$.

Now suppose that $\frac{5m}{2}>t \geq \frac{7m}{3}$.  In this case, Procedure \ref{proc:findH} allows us to subtract copies of $Q$ because $F_t \cdot Q<0$.  From the table, for each copy of $Q$ subtracted, the intersection multiplicity increases by 1; since we can keep subtracting copies of $Q$ as long as the intersection multiplicity with the remaining divisor is negative, we can subtract exactly $-F_t \cdot Q = 5m-2t$ copies. We may also subtract  $\Big \lceil \frac{3m-t}{2} \Big \rceil A$ by the same argument as in the previous case, since subtracting copies of $A$ doesn't change the intersection multiplicity with $Q$ and vice versa.  

Through the process of subtracting $A$s and $Q$s the intersection multiplicities with $C_i$ and $B$ have changed; in particular,  
$$(F_t- \Big \lceil \frac{3m-t}{2} \Big \rceil A - (5m-2t) Q) \cdot C_i  =  t-2m-\Bigg \lceil \frac{3m-t}{2} \Bigg \rceil$$
and 
$$ (F_t- \Big \lceil \frac{3m-t}{2} \Big \rceil A - (5m-2t) Q) \cdot C_i = t-2m-(5m-2t) = 3t-7m.$$
These are both nonnegative, as $t \geq \frac{7m}{3}$, so the intersection multiplicity of the remaining divisor with all elements of $\text{NEG}(X)$ is nonnegative and Procedure \ref{proc:findH} terminates.\footnote{$\frac{7m}{3}$ is always an integer under the divisibility assumption so we don't have to worry about $t$ being the smallest odd integer less than $\frac{7m+1}{3}$.}  Therefore, when $t$ is even
$$H_t = (6t-13m; t-2m, \frac{5t-11m}{2}, \dots, 2t-4m)$$
and $h^0(X, F_t) = 3t^2-13tm+14m^2+\frac{5}{2}t-\frac{11}{2}m+1$, while when $t$ is odd 
$$H_t =  (6t-13m-1; t-2m-1, \frac{5t-11m-1}{2}, \dots, 2t-4m)$$
and $h^0(X, F_t) = 3t^2-13tm+14m^2+\frac{5}{2}t-\frac{11}{2}m+\frac{1}{2}$.

Now suppose that $t = \frac{7m}{3}-1$.  By the same arguments as above, we can subtract $\big \lceil \frac{3m-t}{2} \big \rceil = \frac{2m}{6}$ copies of $A$ and $5m-2t = \frac{m}{3}+2$ copies of $Q$ when following Procedure \ref{proc:findH}.
Then
$$F_{\frac{7m}{3}-1} -\frac{2m}{6}A-\Big(\frac{m}{3}+2\Big)Q = \Big(m-7; \frac{m}{3}-2, \frac{m}{3}-3, \dots, \frac{2m}{3}-2\Big)$$
has intersection multiplicity 1 with $A_i$ and -2 with $C_i$.  At this point, Procedure \ref{proc:findH}  allows us to do the following.
\begin{itemize}
\item Subtract one copy of $C$. Now the intersection multiplicity with $A_i$ is $-1$ and the intersection multiplicity with $C_i$ is $-1$.
\item Subtract one copy of $A$.  Now the intersection multiplicity with $A$ is $1$ and the intersection multiplicity with $C_i$ is $-2$.  
\end{itemize}
It is clear that we can repeat this process as many times as we wish when we follow the procedure; eventually, we will end up with a divisor that has a negative $E_0$ coefficient.  We have that $H_{I^{(m)}}(t) =h^0(X, F_t) = 0$ when $t = \frac{7m}{3}-1$ and thus $H_{I^{(m)}}(t)=0$ for all $t < \frac{7m}{3}$.

%%%%Step 2%%%%

\noindent \textbf{Step 2.}

Assume that $12 | m$.

Now we will turn our attention to the generic initial ideals of $I^{(m)}$.  We compute the number of generators of $\text{gin}(I^{(m)})$ in each degree using Lemma \ref{lem:NumberInDegree} and the Hilbert function values from Step 1.  We have the following.

\begin{center}

\begin{tabular}{|c|c|}
\hline
Value of $t$ & Number of generators of degree $t$\\
\hline
$t< \frac{7m}{3}$ & 0\\
$t = \frac{7m}{3}$ & $\frac{1}{3}m+1$\\
$t = \frac{7m}{3}+1$ & $\frac{2}{3}m+3$\\
$\frac{5m}{2} >t \geq \frac{7m}{3}+2$, $t$ even & $6$\\
$\frac{5m}{2} >t \geq \frac{7m}{3}+2$, $t$ odd & $4$\\
$t=\frac{5m}{2}$ & $6$\\
$t=\frac{5m}{2}+1$& $1$\\
$3m >t \geq \frac{5m}{2}+2$, $t$ even & $2$\\
$3m >t \geq \frac{5m}{2}+2$, $t$ odd & $0$\\
$t =3m$ & $2$\\
$t > 3m$ & 0\\
\hline
\end{tabular}
\end{center}

\noindent \textbf{Step 3.}

Assume once again that $12 | m$.

Note that there are 
$$\frac{\frac{5m}{2}- \frac{7m}{3}-2}{2} = \frac{\frac{m}{6}-2}{2} = \frac{m}{12}-1$$
even (or odd) integers $t$ such that $\frac{5m}{2} >t \geq \frac{7m}{3}+2$, and 
$$\frac{3m-\frac{5m}{2}-2}{2} = \frac{m}{4}-1$$
even (or odd) integers $t$ such that $3m >t \geq \frac{5m}{2}+2$.

Using the results of Step 2, we can find strictly decreasing $\lambda_i$ such that
$$\text{gin}(I^{(m)}) = (x^{k}, x^{k-1}y^{\lambda_{k-1}}, \dots, xy^{\lambda_1}, y^{\lambda_0}).$$
Since the smallest degree generator is of degree $\frac{7m}{3}$, $k = \frac{7m}{3}$.

The values of $\lambda_i$ that we obtain are shown in the following table.

\newpage

\begin{center}
\begin{adjustwidth}{-.5in}{-.5in}% adjust the L and R margins by 1 inch

\begin{tabular}{|r | c c c c : c c c c : c c c |}
\hline
degree & $\frac{7m}{3}$ & & & & $\frac{7m}{3}+1$ & & & & $\frac{7m}{3}+2$ & & \\
$i$ & $\frac{7m}{3}$ & $\frac{7m}{3}-1$ & $\cdots$& $2m$ & $2m-1$ & $2m-2$ & $\cdots$ & $\frac{4m}{3}-3$ & $\frac{4m}{3}-4$& $\cdots$ & $\frac{4m}{3}-9$ \\
$\lambda_i$ &0 & 1 & $\cdots$ & $\frac{m}{3}$ & $\frac{m}{3}+2$ & $\frac{m}{3}+3$ & $\cdots$ & $m+4$ & $m+6$ & $\cdots$ & $m+11$ \\
\hline
\end{tabular}

\begin{tabular}{|r | c c c : c : c c c c |}
\hline
degree & $\frac{7m}{3}+3$& & & $\cdots$ & $\frac{5m}{2}$ & & & \\
$i$ & $\frac{4m}{3}-10$ & $\cdots$ & $\frac{4m}{3}-13$ & $\cdots$ & $\frac{4m}{3}-4-10(\frac{m}{12}-1) = \frac{m}{2}+6$ & $\frac{m}{2}+5$ & $\cdots$ & $\frac{m}{2}+1$ \\
$\lambda_i$ & $m+13$ & $\cdots$ & $m+16$ & $\cdots$ & $m+6+12(\frac{m}{12}-1)=2m-6$ & $2m-5$ & $\cdots$& $2m-1$ \\
\hline
\end{tabular}

\begin{tabular}{|r | c : c c : c c :  c : c  c|}
\hline
degree & $\frac{5m}{2}+1$ & $\frac{5m}{2}+2$ & & $\frac{5m}{2}+4$ & & $\cdots$ & $3m$ & \\
$i$ & $\frac{m}{2}$ & $\frac{m}{2} -1$ & $\frac{m}{2}-2$ & $\frac{m}{2}-3$ & $\frac{m}{2}-4$& $\cdots$ & $\frac{m}{2}-1-2(\frac{m}{4}-1) = 1$& 0\\
$\lambda_i$ & $2m+1$ & $2m+3$ & $2m+4$ & $2m+7$ & $2m+8$ & $\cdots$ & $2m+3+4(\frac{m}{4}-1) = 3m-1$ & $3m$\\
\hline
\end{tabular}

\end{adjustwidth}
\end{center}

\vspace{0.1in}

\noindent \textbf{Step 4.}

Assume that $12 | m$.

The Newton polytope of $\text{gin}(I^{(m)})$ is the convex hull of the ideal when thought of as a subset of $\mathbb{R}^2$.  In particular, its boundary is determined by the points $(i, \lambda_i)$ recorded in the table from Step 3.  Plotting these points, one can see that the boundary of $P_{\text{gin}(I^{(m)})}$ is defined by the line segments through the points $(0,3m)$, $(\frac{m}{2}-2, 2m+4)$, $(\frac{m}{2}+1, 2m-1)$, $(\frac{4m}{3}-9,m+11)$, $(\frac{4m}{3}-3, m+4)$, $(2m, \frac{m}{3})$, and $(\frac{7m}{3},0)$.

\noindent \textbf{Step 5.}

Scaling $P_{\text{gin}(I^{(m)})}$ from the previous step by $\frac{1}{m}$ and taking the limit as $m$ approaches infinity, the limiting shape of the symbolic generic initial system is defined by the line segments through the following points. 
$$(0, 3) = \lim_{m \rightarrow \infty} \Big( 0, \frac{3m}{m} \Big)$$

$$\Big( \frac{1}{2}, 2 \Big) = \lim_{m \rightarrow \infty} \Big( \frac{m/2-2}{m}, \frac{2m+4}{m} \Big) =  \lim_{m \rightarrow \infty} \Big( \frac{m/2+1}{m}, \frac{2m-1}{m} \Big)$$

$$\Big( \frac{4}{3}, 1\Big) =  \lim_{m \rightarrow \infty} \Big( \frac{4m/3-9}{m}, \frac{m+11}{m} \Big) = \lim_{m \rightarrow \infty} \Big( \frac{4m/3-3}{m}, \frac{m+4}{m} \Big)$$

$$\Big( 2, \frac{1}{3} \Big) =  \lim_{m \rightarrow \infty} \Big( \frac{2m}{m}, \frac{m/3}{m} \Big)$$

$$\Big( \frac{7}{3}, 0\Big) = \lim_{m \rightarrow \infty} \Big( \frac{7m/3}{m}, 0 \Big)$$
Note that $(2, \frac{1}{3})$ lies on the line segment connecting $(\frac{4}{3},1)$ with $(\frac{7}{3},0)$ so it is not a vertex of the boundary of the limiting shape.
\section{Point Configurations and Limiting Shapes:  Questions and Observations}
\label{sec:questions}

In this section we investigate how the arrangement of points in a point configuration influences the limiting shape of the symbolic generic initial system of the corresponding ideal.  Throughout $I$ will be the ideal of a point configuration in $\mathbb{P}^2$ and $P$ will denote the limiting shape of $\{ \text{gin}(I^{(m)})\}_m$.

The following is Lemma 2.5 of \cite{Mayes12c} and is proven there; it describes how the number of points in a configuration is reflected in the limiting shape.

\begin{lem}
\label{lem:areaunder}
Let $I$ be the ideal corresponding to an arrangement of $r$ distinct points in $\mathbb{P}^2$.  If $Q$ is the complement of the limiting shape $P$ of $\{ \textnormal{gin}(I^{(m)})\}_m$ in $\mathbb{R}_{\geq 0}^2$, then the area of $Q$ is equal to $\frac{r}{2}$.
\end{lem}

While this lemma imposes strong restrictions on where the limiting shape can lie, we would like a more precise description in terms of geometry.

\begin{question}
\label{question:intercepts}
What is the meaning of the intercepts of the boundary of $P$?  Can one see these intercepts in the point configuration?
\end{question}

We may partially answer this question.  The $x$-intercept of $P$ is equal to $\lim_{m\rightarrow \infty} \frac{\alpha(m)}{m}$ where $\alpha(m)$ is the degree of the smallest degree element of $I^{(m)}$.  This limit has been studied for some special point configurations and is sometimes equal to the \textit{Seshadri constant} of $I$ (see, for example, \cite{BC10} and \cite{Harbourne02}).  Unfortunately, it does not seem as though there is a simple connection to the point configuration.

The $y$-intercept of the boundary of $P$ is equal to $\lim_{m\rightarrow \infty} \frac{\text{reg}(I^{(m)})}{m}$ (see Lemma 3.1 of \cite{Mayes12c}).  This limit is not as well-studied as the previous one, but appears to have a nice geometric meaning in certain cases.  For example, when there is a line passing through at least three points and algorithms similar to the one outlined in Section \ref{sec:HilbFn} may be used to find the Hilbert function of $I^{(m)}$,  $\lim_{m\rightarrow \infty} \frac{\text{reg}(I^{(m)})}{m}$ is equal to the maximum number of points lying on a single line.

\begin{question}
\label{question:shape}
What features does a point configuration possess when the boundary of $P$ consists of a fixed number of line segments?
\end{question}

To describe one potential answer to Question \ref{question:shape}, we will distinguish between different `types' of points within a configuration.  A curve of degree $d$ \textit{defines} a point configuration if at least ${d+2 \choose 2}$ points in the configuration lie on the curve.  If $N$ points of the configuration lie on such a curve, we will denote this curve by $C_{d,N}$.  For example, in Configuration $H$ shown in Figure \ref{fig:6HArrangement}, there are two curves defining the point configuration.  Since they are both lines containing three points, the set of curves defining the configuration is denoted $\{ C_{1,3}, C_{1,3} \}$.  

Each point within a configuration may then be associated with the set of curves defining the configuration that pass through that point.  For example, in Figure \ref{fig:6HArrangement}, points 2, 3, 4, and 5 correspond to the set $\{ C_{1,3}\}$, point 1 corresponds to the set $\{ C_{1,3}, C_{1,3}\}$, and point $6$ corresponds to the empty set.  We will call such sets the \textit{incidence type} of a point.  Thus, Configuration $H$ has three distinct incidence types.  Configuration $F$ also has three distinct incidence types:  there are two points of incidence type $\{ C_{1,3} \}$, three of type $\{ C_{1,4} \}$, and one of type $\{ C_{1,3}, C_{1,4} \}$.  Configuration $I$ has two distinct incidence types.

\begin{observation}
Suppose that $I$ is the ideal corresponding to one of the following subsets of $\mathbb{P}^2$:  a point configuration of at most six points; a point configuration arising from a complete intersection;  a generic set of points;  points on an irreducible conic; a point configuration where all but one point lies on a line; a point configuration where all but two points lies on a line and no other line passes through three points; or a star point configuration.  Then the number of line segments forming the boundary of the limiting shape of $\{ \text{gin}(I^{(m)})\}_m$ is equal to the number of distinct incidence types of the points in the corresponding point configuration (\cite{Mayes12c}, \cite{Mayes12d}, \cite{Mayes12e}, \cite{Mayes12a}).
\end{observation}

While we do not have enough evidence to claim that the answer to Question \ref{question:shape} is always given by the number of incidence types in a configuration, it is interesting to note that it holds for all of the cases that have been studied up to this point.  This provides further evidence that our asymptotic viewpoint reveals information that cannot be seen by looking at the Hilbert functions of individual fat point ideals $I^{(m)}$.

\begin{question}
Is there a geometric interpretation for the coordinates of the `crux points' lying on the intersection of the line segments defining the boundary of $P$?
\end{question}

The answer to this final question seems mysterious; it is likely that many more configurations will need to be studied to formulate a reasonable conjecture to answer this question.

\bibliography{SixPointBib}
\bibliographystyle{amsalpha}
\nocite{*}

\end{document}